\documentclass[12pt]{article}

\usepackage{amsthm,amsmath,amssymb,amsfonts}
\usepackage{geometry}
\usepackage{mathtools}
\usepackage{microtype}
\usepackage[shortlabels]{enumitem}
\usepackage{graphicx}
\usepackage{tikz}
\usepackage[dvipsnames]{xcolor}
\usepackage[english]{babel}

\usepackage[colorlinks=true,linkcolor=blue,citecolor=blue,urlcolor=blue,
  pdftitle={Superlinear Lower Bounds for Monochromatic Path Partitions},
  pdfauthor={Heng Li and Lanchao Wang}]{hyperref}
\usepackage[capitalise,noabbrev]{cleveref}

\newtheorem{theorem}{Theorem}[section]
\newtheorem{lemma}[theorem]{Lemma}
\newtheorem{proposition}[theorem]{Proposition}

\newtheorem{conjecture}[theorem]{Conjecture}

\newtheorem{problem}[theorem]{Problem}
\newtheorem{construction}[theorem]{Construction}

\crefname{lemma}{Lemma}{Lemmas}
\crefname{proposition}{Proposition}{Propositions}
\crefname{observation}{Observation}{Observations}
\crefname{conjecture}{Conjecture}{Conjectures}
\crefname{corollary}{Corollary}{Corollaries}
\crefname{problem}{Problem}{Problems}
\crefname{construction}{Construction}{Constructions}

\Crefname{problem}{Problem}{Problems}
\Crefname{construction}{Construction}{Constructions}

\newcommand{\eps}{\varepsilon}
\newcommand{\im}{\operatorname{im}}
\newcommand{\ppath}{p_{\mathrm{path}}}
\newcommand{\pcyc}{p_{\mathrm{cyc}}}
\newcommand{\bpath}{b_{\mathrm{path}}}
\newcommand{\bcyc}{b_{\mathrm{cyc}}}

\title{\fontsize{19}{22}\selectfont Superlinear Lower Bounds for Monochromatic Path Partitions}

\author{
Heng Li\thanks{School of Mathematics, Shandong University, Jinan, China, and
Extremal Combinatorics and Probability Group (ECOPRO), Institute for Basic
Science (IBS), Daejeon, South Korea. Email: heng.li@sdu.edu.cn.}
\and
Lanchao Wang\thanks{School of Mathematics, Nanjing University, Nanjing, China, and
Extremal Combinatorics and Probability Group (ECOPRO), Institute for Basic
Science (IBS),  Daejeon, South Korea.
Email: lanchaowang@foxmail.com.}
}

\date{}

\begin{document}

\maketitle

\begin{abstract}
In 1989, Gy\'arf\'as conjectured that the vertex set of every
$r$-edge-coloured complete graph can be partitioned into at most $r$
vertex-disjoint monochromatic paths. Erd\H{o}s, Gy\'arf\'as, and Pyber
subsequently proposed the analogous conjecture for monochromatic cycles.
Pokrovskiy proved Gy\'arf\'as's conjecture for $r=3$, while disproving
the conjecture of Erd\H{o}s, Gy\'arf\'as, and Pyber for every $r\ge3$
by constructing colourings that require at least $r+1$ monochromatic
cycles. In this paper, we
disprove Gy\'arf\'as's conjecture in a quantitatively strong superlinear
form: for every sufficiently large $r$, there exists an $r$-edge-coloured
complete graph that requires at least $(1-o(1))r\log\log r$
vertex-disjoint monochromatic paths. Consequently, the monochromatic
cycle-partition number is also superlinear in $r$.
Our construction also disproves two conjectures of Pokrovskiy: one on
monochromatic cycle coverings and the other on path coverings in the
balanced bipartite setting.
\end{abstract}
\section{Introduction}

Monochromatic partition problems ask for a bounded number of
vertex-disjoint monochromatic pieces covering all vertices. Paths and cycles are the basic cases. This line of research goes back
to Gerencs\'er and Gy\'arf\'as~\cite{GerencserGyarfas}, who proved that every $2$-edge-coloured complete graph can be partitioned
into two monochromatic paths. In 1989, Gy\'arf\'as~\cite{Gyarfas} posed the following conjecture.

\begin{conjecture}\label{conj}
For every $r\ge3$, the vertex set of every $r$-edge-coloured complete
graph can be partitioned into at most $r$ vertex-disjoint monochromatic
paths.
\end{conjecture}

The conjecture was motivated by a theorem of Rado~\cite{Rado}, who proved
the corresponding statement for countably infinite complete graphs. This conjecture has prompted considerable subsequent work; see, for
example, the survey of Gy\'arf\'as~\cite{GyarfasSurvey}. For
convenience, we first introduce some notation.

For $r\ge2$, let $\ppath(r)$ be the minimum number of vertex-disjoint monochromatic paths needed to partition the vertex set of any $r$-edge-coloured complete graph, and define $\pcyc(r)$ analogously for monochromatic cycles, where single
vertices and edges are also regarded as cycles. The finiteness of $\ppath(r)$ was posed by Gy\'arf\'as as a
weaker form of Conjecture~\ref{conj}. Erd\H{o}s, Gy\'arf\'as, and Pyber~\cite{EGP} resolved it by proving
that $\pcyc(r)=O(r^2\log r)$, which also implies
$\ppath(r)=O(r^2\log r)$ since every monochromatic cycle has a
monochromatic spanning path. Erd\H{o}s, Gy\'arf\'as, and Pyber further conjectured that
$\pcyc(r)\le r$, with the case $r=2$ having been posed earlier by Lehel.

Gy\'arf\'as, Ruszink\'o, S\'ark\"ozy, and Szemer\'edi~\cite{GRSS}
proved that, for every fixed $r$, every sufficiently large
$r$-edge-coloured complete graph can be partitioned into at most
$100r\log r$ monochromatic cycles. In the three-colour case, they
further showed~\cite{GRSS11} that all but $o(n)$ vertices can be
partitioned into three monochromatic cycles. Lehel's conjecture was established through work of {\L}uczak, R\"odl,
and Szemer\'edi~\cite{LRS} and Allen~\cite{Allen}, and was finally
resolved by Bessy and Thomass\'e~\cite{BT}. Pokrovskiy~\cite{Pok} proved that $\ppath(3)\le3$, settling
Gy\'arf\'as's conjecture for three colours. In the same paper, he
disproved the conjecture of Erd\H{o}s, Gy\'arf\'as, and Pyber for every
$r\ge3$ by constructing $r$-edge-coloured complete graphs that require
at least $r+1$ monochromatic cycles. This naturally raised the questions of whether Gy\'arf\'as's conjecture
holds for all $r$ and whether $\pcyc(r)=O(r)$. In this paper, we show that both statements fail and, in fact, obtain a
superlinear lower bound.

\begin{theorem}\label{thm:main}
For every sufficiently large $r$, we have
$\ppath(r)\ge (1-o(1))r\log\log r.$
\end{theorem}

The key idea is to construct a random colouring in which each
monochromatic subgraph splits into many well-controlled components.
Since every monochromatic path lies in a single component, a partition
into a few paths can involve only a few such components, which we show
cannot together cover all vertices. Taking suitable blow-ups of the construction yields
counterexamples with an arbitrarily large number of vertices.

\begin{theorem}\label{cor:arbitrarily-large}
For every sufficiently large $r$ and every $n_0$, there exist $n>n_0$
and an $r$-edge-colouring of $K_n$ whose vertex set cannot be partitioned
into fewer than $(1-o(1))r\log\log r$ monochromatic paths.
\end{theorem}

After disproving the conjecture of Erd\H{o}s, Gy\'arf\'as, and Pyber~\cite{EGP},
Pokrovskiy~\cite{Pok} conjectured the weaker statement that, for every
fixed $r$, there is a constant $c_r$ such that $r$ vertex-disjoint
monochromatic cycles cover all but at most $c_r$ vertices in every
$r$-edge-coloured complete graph. Our construction gives the following stronger statement for paths,
which also disproves this conjecture.

\begin{theorem}\label{thm:path-uncovered}
For every sufficiently large integer $r$, there exists a constant
$\eta_r>0$ such that, for infinitely many integers $n$, there is an
$r$-edge-colouring of $K_n$ in which every collection of at most $r$
vertex-disjoint monochromatic paths leaves at least $\eta_r n$
vertices uncovered.
\end{theorem}

Monochromatic path and cycle partition problems have also been studied
for non-complete host graphs, including nearly complete graphs, graphs
with prescribed independence number or large minimum degree, and random
graphs; see, for example,~\cite{GJS97,Sarkozy11,BBGGS,Letzter19,KLLP,LangLo,DiBraccioPatel}. Our method also extends to balanced complete bipartite graphs, whose
monochromatic partition problems go back to Gy\'arf\'as~\cite{Gyarfas83}.
For $r\ge2$, let $\bpath(r)$ be the minimum number of vertex-disjoint
monochromatic paths needed to partition the vertex set of any
$r$-edge-coloured balanced complete bipartite graph, and define
$\bcyc(r)$ analogously for monochromatic cycles.

For cycles, Haxell~\cite{Haxell} proved
$\bcyc(r)=O((r\log r)^2)$, later improved to
$O(r^2\log r)$ by Peng, R\"odl, and Ruci\'nski~\cite{PRR}.
For fixed $r$ and sufficiently large $n$, Lang and Stein~\cite{LS}
showed that $4r^2$ monochromatic cycles suffice to partition
$K_{n,n}$. The cases of two and three colours were studied by
Schaudt and Stein~\cite{SchaudtStein} and by Lang, Schaudt, and
Stein~\cite{LangSchaudtStein}, while Benevides, Quintino, and
Talon~\cite{BQT} recently proved that $\bcyc(2)\le4$.

For paths, Pokrovskiy~\cite{Pok} proved $\bpath(2)=3$ and, more
generally, $\bpath(r)\ge2r-1$ for every $r\ge2$. He conjectured that
$\bpath(r)\le2r-1$ for all $r\ge2$.
We disprove this conjecture.

\begin{theorem}\label{thm:bip}
We have $\bpath(r)=\Omega(r\log r)$.
\end{theorem}

Our construction also yields a consequence for partitions into monochromatic
regular graphs. For $k\ge2$, let $p_k(r)$ be the minimum number of connected
monochromatic $k$-regular graphs and single vertices required to
partition any $r$-edge-coloured complete graph. S\'ark\"ozy and
Selkow~\cite{SS} proved that
$p_k(r)\le r^{O(r\log r+k)}$.
S\'ark\"ozy, Selkow, and Song~\cite{SSS} later showed that, for
sufficiently large complete graphs, $100r\log r+2rk$ pieces suffice,
and also gave the lower bound $(r-1)(k-1)+1$. Thus, for fixed $k$, the
previously known lower bound was only linear in $r$. Our argument yields
a superlinear lower bound.

\begin{proposition}\label{prop:regular}
For every sufficiently large $r$ and every integer $k\ge2$,
$$
p_k(r)\ge (1-o(1))r\log\log r,
$$
where the $o(1)$ depends only on $r$.
\end{proposition}

The rest of the paper is organised as follows.
Section~\ref{sec:proofoverview} gives an overview of the proof of
Theorem~\ref{thm:main}. In Section~\ref{sec:complete}, we prove
Theorems~\ref{thm:main}, \ref{cor:arbitrarily-large}, and
\ref{thm:path-uncovered}, together with Proposition~\ref{prop:regular}.
We prove Theorem~\ref{thm:bip} in Section~\ref{sec:bipartite}, and
conclude with several open problems.

\section{Proof overview}\label{sec:proofoverview}
In this section, we give an overview of the proof of
Theorem~\ref{thm:main}. The proof has two parts. We first construct
the colouring using random class-to-colour assignments together with a
sparse auxiliary colouring that separates the monochromatic pieces. We
then analyse monochromatic matchings in the resulting graph to obtain the
desired lower bound.
\smallskip

\noindent\textbf{Step 1: Constructing the colouring.}
We partition the vertex set $V$ into $m$ equal-sized classes
$V_1,\ldots,V_m$. Fix a total order $\prec$ on $V$ such that
$V_i\prec V_j$ whenever $i<j$. Take $q$ base colours.
Independently for each vertex $v\in V$, choose a uniformly random
injection $\phi_v:[m]\hookrightarrow[q]$. For each edge $uv$ with
$u\prec v$ and $u\in V_i$, associate $uv$ with the base colour
$\phi_v(i)$. Thus, for a fixed vertex $v$, all edges from the same
earlier class $V_i$ to $v$ are associated with the same colour
$\phi_v(i)$, while edges from different earlier classes are associated
with different colours.

For each $c\in[q]$, split the edges receiving colour $c$ into $m$
pieces according to the class of their earlier endpoint. These pieces
need not be vertex-disjoint: a vertex may be a later endpoint in one
piece and an earlier endpoint in another. We can show that, by recolouring the edges of a sparse auxiliary graph
$J$ with $o(q)$ fresh colours so that they are properly edge-coloured,
these pieces can be made pairwise vertex-disjoint. For each
$i\in[m]$, let $F_{i,c}$ be the graph formed by the
remaining edges in the $i$th piece. Hence, for each base colour $c$, every nontrivial connected colour-$c$
subgraph lies in a single $F_{i,c}$; see
Figure~\ref{fig:construction-overview}.

\medskip
\noindent\textbf{Step 2: Analysing matchings in a path partition.} 
Let $\mathcal P$ be a partition of $V$ into $\ell$ monochromatic paths.
Since the graphs $F_{1,c},\ldots,F_{m,c}$ are pairwise
vertex-disjoint, every nontrivial path in $\mathcal P$ with base colour $c$ lies in a
unique $F_{i,c}$. Thus each base-coloured path uses
at most one pair $(i,c)$. Choose a maximum matching from every path in
$\mathcal P$, and let $M$ be their union. Then $M$ leaves at most
$\ell$ vertices unmatched.

For each $i$, let $C_i(M)$ be the set of
base colours $c$ for which $M$ contains an edge of $F_{i,c}$. Since each path in $\mathcal P$ contributes to at most
one pair $(i,c)$, we have
$
\sum_{i=1}^m |C_i(M)|\le \ell.
$
Moreover, since $J$ is properly edge-coloured, every
auxiliary colour class is a matching, so every auxiliary-coloured path
contains at most one edge, and hence $M$ contains only a few
auxiliary-coloured edges. Thus, when $\ell$ is small, $M$ covers almost all vertices, its
base-coloured edges are supported on only a few graphs $F_{i,c}$, and it
contains only a few auxiliary-coloured edges.

For a suitable choice of $s$, consider the final $s$ classes
$V_{m-s+1}\cup\cdots\cup V_m$. If a vertex $v$ in these classes is
matched to an earlier class $V_i$ by a base-coloured edge, then
$\phi_v(i)\in C_i(M)$. Since the total size of the sets $C_i(M)$ is
small and the labels $\phi_v$ are random, many vertices in the final
$s$ classes avoid $C_i(M)$ for most earlier indices $i$.  Such vertices
have very few ways to be matched: they cannot be matched to most
earlier classes by base-coloured edges,  and only a few can be
matched by auxiliary-coloured edges, since $M$ contains only a few such
edges. If $\ell$ is too small, there are not enough
remaining edges to match all of them, contradicting the fact that $M$
leaves only a few vertices unmatched.

\begin{figure}[ht]
    \centering

\tikzset{every picture/.style={line width=0.75pt}} 

\begin{tikzpicture}[x=0.75pt,y=0.75pt,yscale=-0.9,xscale=0.9]

\draw  [fill={rgb, 255:red, 80; green, 227; blue, 194 }  ,fill opacity=0.5 ] (417,161.1) .. controls (417,154.97) and (421.97,150) .. (428.1,150) -- (461.4,150) .. controls (467.53,150) and (472.5,154.97) .. (472.5,161.1) -- (472.5,260.4) .. controls (472.5,266.53) and (467.53,271.5) .. (461.4,271.5) -- (428.1,271.5) .. controls (421.97,271.5) and (417,266.53) .. (417,260.4) -- cycle ;
\draw [color={rgb, 255:red, 144; green, 19; blue, 254 }  ,draw opacity=1 ]   (431.7,180.8) -- (457.7,198.8) ;
\draw [color={rgb, 255:red, 144; green, 19; blue, 254 }  ,draw opacity=1 ]   (432.2,209.5) -- (458.2,226.3) ;
\draw [color={rgb, 255:red, 144; green, 19; blue, 254 }  ,draw opacity=1 ]   (431.7,209.8) -- (457.7,199) ;
\draw [color={rgb, 255:red, 144; green, 19; blue, 254 }  ,draw opacity=1 ]   (431.7,237.8) -- (458.2,226.3) ;
\draw [color={rgb, 255:red, 144; green, 19; blue, 254 }  ,draw opacity=1 ]   (431.7,180.8) -- (458.2,225.1) ;
\draw [color={rgb, 255:red, 144; green, 19; blue, 254 }  ,draw opacity=1 ]   (431.7,237.8) -- (457.7,198.8) ;
\draw [color={rgb, 255:red, 155; green, 155; blue, 155 }  ,draw opacity=1 ]   (299.7,164.6) -- (299.7,184.8) ;
\draw [color={rgb, 255:red, 248; green, 231; blue, 28 }  ,draw opacity=1 ]   (206.3,241.2) -- (299.7,184.8) ;
\draw [color={rgb, 255:red, 248; green, 231; blue, 28 }  ,draw opacity=1 ]   (206.3,253.2) -- (299.7,184.8) ;
\draw [color={rgb, 255:red, 144; green, 19; blue, 254 }  ,draw opacity=1 ]   (136.7,241.2) -- (299.7,184.8) ;
\draw [color={rgb, 255:red, 144; green, 19; blue, 254 }  ,draw opacity=1 ][fill={rgb, 255:red, 144; green, 19; blue, 254 }  ,fill opacity=1 ]   (136.7,253.2) -- (299.7,184.8) ;
\draw [color={rgb, 255:red, 208; green, 2; blue, 27 }  ,draw opacity=1 ][fill={rgb, 255:red, 208; green, 2; blue, 27 }  ,fill opacity=1 ]   (66.3,253.2) -- (299.7,184.8) ;
\draw [color={rgb, 255:red, 208; green, 2; blue, 27 }  ,draw opacity=1 ][fill={rgb, 255:red, 208; green, 2; blue, 27 }  ,fill opacity=1 ]   (66.1,239.6) -- (299.7,184.8) ;
\draw  [fill={rgb, 255:red, 80; green, 227; blue, 194 }  ,fill opacity=0.5 ] (47.5,157.28) .. controls (47.5,153.26) and (50.76,150) .. (54.78,150) -- (76.62,150) .. controls (80.64,150) and (83.9,153.26) .. (83.9,157.28) -- (83.9,264.22) .. controls (83.9,268.24) and (80.64,271.5) .. (76.62,271.5) -- (54.78,271.5) .. controls (50.76,271.5) and (47.5,268.24) .. (47.5,264.22) -- cycle ;
\draw  [fill={rgb, 255:red, 80; green, 227; blue, 194 }  ,fill opacity=0.5 ] (117.5,157.28) .. controls (117.5,153.26) and (120.76,150) .. (124.78,150) -- (146.62,150) .. controls (150.64,150) and (153.9,153.26) .. (153.9,157.28) -- (153.9,264.22) .. controls (153.9,268.24) and (150.64,271.5) .. (146.62,271.5) -- (124.78,271.5) .. controls (120.76,271.5) and (117.5,268.24) .. (117.5,264.22) -- cycle ;
\draw  [fill={rgb, 255:red, 80; green, 227; blue, 194 }  ,fill opacity=0.5 ] (187.5,157.28) .. controls (187.5,153.26) and (190.76,150) .. (194.78,150) -- (216.62,150) .. controls (220.64,150) and (223.9,153.26) .. (223.9,157.28) -- (223.9,264.22) .. controls (223.9,268.24) and (220.64,271.5) .. (216.62,271.5) -- (194.78,271.5) .. controls (190.76,271.5) and (187.5,268.24) .. (187.5,264.22) -- cycle ;
\draw  [fill={rgb, 255:red, 80; green, 227; blue, 194 }  ,fill opacity=0.5 ] (281.5,157.28) .. controls (281.5,153.26) and (284.76,150) .. (288.78,150) -- (310.62,150) .. controls (314.64,150) and (317.9,153.26) .. (317.9,157.28) -- (317.9,264.22) .. controls (317.9,268.24) and (314.64,271.5) .. (310.62,271.5) -- (288.78,271.5) .. controls (284.76,271.5) and (281.5,268.24) .. (281.5,264.22) -- cycle ;
\draw  [fill={rgb, 255:red, 0; green, 0; blue, 0 }  ,fill opacity=1 ] (238.3,216.65) .. controls (238.3,216.4) and (238.79,216.2) .. (239.41,216.2) .. controls (240.02,216.2) and (240.51,216.4) .. (240.51,216.65) .. controls (240.51,216.9) and (240.02,217.1) .. (239.41,217.1) .. controls (238.79,217.1) and (238.3,216.9) .. (238.3,216.65) -- cycle ;
\draw  [fill={rgb, 255:red, 0; green, 0; blue, 0 }  ,fill opacity=1 ] (251.19,216.65) .. controls (251.19,216.4) and (251.69,216.2) .. (252.3,216.2) .. controls (252.91,216.2) and (253.41,216.4) .. (253.41,216.65) .. controls (253.41,216.9) and (252.91,217.1) .. (252.3,217.1) .. controls (251.69,217.1) and (251.19,216.9) .. (251.19,216.65) -- cycle ;
\draw  [fill={rgb, 255:red, 0; green, 0; blue, 0 }  ,fill opacity=1 ] (264.09,216.65) .. controls (264.09,216.4) and (264.58,216.2) .. (265.19,216.2) .. controls (265.81,216.2) and (266.3,216.4) .. (266.3,216.65) .. controls (266.3,216.9) and (265.81,217.1) .. (265.19,217.1) .. controls (264.58,217.1) and (264.09,216.9) .. (264.09,216.65) -- cycle ;
\draw  [fill={rgb, 255:red, 0; green, 0; blue, 0 }  ,fill opacity=1 ] (65.1,253.2) .. controls (65.1,252.54) and (65.64,252) .. (66.3,252) .. controls (66.96,252) and (67.5,252.54) .. (67.5,253.2) .. controls (67.5,253.86) and (66.96,254.4) .. (66.3,254.4) .. controls (65.64,254.4) and (65.1,253.86) .. (65.1,253.2) -- cycle ;
\draw  [fill={rgb, 255:red, 0; green, 0; blue, 0 }  ,fill opacity=1 ] (298.5,164.8) .. controls (298.5,164.14) and (299.04,163.6) .. (299.7,163.6) .. controls (300.36,163.6) and (300.9,164.14) .. (300.9,164.8) .. controls (300.9,165.46) and (300.36,166) .. (299.7,166) .. controls (299.04,166) and (298.5,165.46) .. (298.5,164.8) -- cycle ;
\draw  [fill={rgb, 255:red, 0; green, 0; blue, 0 }  ,fill opacity=1 ] (64.9,239.6) .. controls (64.9,238.94) and (65.44,238.4) .. (66.1,238.4) .. controls (66.76,238.4) and (67.3,238.94) .. (67.3,239.6) .. controls (67.3,240.26) and (66.76,240.8) .. (66.1,240.8) .. controls (65.44,240.8) and (64.9,240.26) .. (64.9,239.6) -- cycle ;
\draw  [fill={rgb, 255:red, 0; green, 0; blue, 0 }  ,fill opacity=1 ] (135.5,253.2) .. controls (135.5,252.54) and (136.04,252) .. (136.7,252) .. controls (137.36,252) and (137.9,252.54) .. (137.9,253.2) .. controls (137.9,253.86) and (137.36,254.4) .. (136.7,254.4) .. controls (136.04,254.4) and (135.5,253.86) .. (135.5,253.2) -- cycle ;
\draw  [fill={rgb, 255:red, 0; green, 0; blue, 0 }  ,fill opacity=1 ] (135.5,241.2) .. controls (135.5,240.54) and (136.04,240) .. (136.7,240) .. controls (137.36,240) and (137.9,240.54) .. (137.9,241.2) .. controls (137.9,241.86) and (137.36,242.4) .. (136.7,242.4) .. controls (136.04,242.4) and (135.5,241.86) .. (135.5,241.2) -- cycle ;
\draw  [fill={rgb, 255:red, 0; green, 0; blue, 0 }  ,fill opacity=1 ] (205.1,253.2) .. controls (205.1,252.54) and (205.64,252) .. (206.3,252) .. controls (206.96,252) and (207.5,252.54) .. (207.5,253.2) .. controls (207.5,253.86) and (206.96,254.4) .. (206.3,254.4) .. controls (205.64,254.4) and (205.1,253.86) .. (205.1,253.2) -- cycle ;
\draw  [fill={rgb, 255:red, 0; green, 0; blue, 0 }  ,fill opacity=1 ] (205.1,241.2) .. controls (205.1,240.54) and (205.64,240) .. (206.3,240) .. controls (206.96,240) and (207.5,240.54) .. (207.5,241.2) .. controls (207.5,241.86) and (206.96,242.4) .. (206.3,242.4) .. controls (205.64,242.4) and (205.1,241.86) .. (205.1,241.2) -- cycle ;
\draw  [fill={rgb, 255:red, 0; green, 0; blue, 0 }  ,fill opacity=1 ] (298.5,184.8) .. controls (298.5,184.14) and (299.04,183.6) .. (299.7,183.6) .. controls (300.36,183.6) and (300.9,184.14) .. (300.9,184.8) .. controls (300.9,185.46) and (300.36,186) .. (299.7,186) .. controls (299.04,186) and (298.5,185.46) .. (298.5,184.8) -- cycle ;
\draw  [fill={rgb, 255:red, 0; green, 0; blue, 0 }  ,fill opacity=1 ] (65.1,166.2) .. controls (65.1,165.54) and (65.64,165) .. (66.3,165) .. controls (66.96,165) and (67.5,165.54) .. (67.5,166.2) .. controls (67.5,166.86) and (66.96,167.4) .. (66.3,167.4) .. controls (65.64,167.4) and (65.1,166.86) .. (65.1,166.2) -- cycle ;
\draw  [fill={rgb, 255:red, 0; green, 0; blue, 0 }  ,fill opacity=1 ] (65.1,184.2) .. controls (65.1,183.54) and (65.64,183) .. (66.3,183) .. controls (66.96,183) and (67.5,183.54) .. (67.5,184.2) .. controls (67.5,184.86) and (66.96,185.4) .. (66.3,185.4) .. controls (65.64,185.4) and (65.1,184.86) .. (65.1,184.2) -- cycle ;
\draw  [fill={rgb, 255:red, 0; green, 0; blue, 0 }  ,fill opacity=1 ] (135.6,167.2) .. controls (135.6,166.54) and (136.14,166) .. (136.8,166) .. controls (137.46,166) and (138,166.54) .. (138,167.2) .. controls (138,167.86) and (137.46,168.4) .. (136.8,168.4) .. controls (136.14,168.4) and (135.6,167.86) .. (135.6,167.2) -- cycle ;
\draw  [fill={rgb, 255:red, 0; green, 0; blue, 0 }  ,fill opacity=1 ] (135.6,184.7) .. controls (135.6,184.04) and (136.14,183.5) .. (136.8,183.5) .. controls (137.46,183.5) and (138,184.04) .. (138,184.7) .. controls (138,185.36) and (137.46,185.9) .. (136.8,185.9) .. controls (136.14,185.9) and (135.6,185.36) .. (135.6,184.7) -- cycle ;
\draw  [fill={rgb, 255:red, 0; green, 0; blue, 0 }  ,fill opacity=1 ] (205.1,165.2) .. controls (205.1,164.54) and (205.64,164) .. (206.3,164) .. controls (206.96,164) and (207.5,164.54) .. (207.5,165.2) .. controls (207.5,165.86) and (206.96,166.4) .. (206.3,166.4) .. controls (205.64,166.4) and (205.1,165.86) .. (205.1,165.2) -- cycle ;
\draw  [fill={rgb, 255:red, 0; green, 0; blue, 0 }  ,fill opacity=1 ] (205.1,184.2) .. controls (205.1,183.54) and (205.64,183) .. (206.3,183) .. controls (206.96,183) and (207.5,183.54) .. (207.5,184.2) .. controls (207.5,184.86) and (206.96,185.4) .. (206.3,185.4) .. controls (205.64,185.4) and (205.1,184.86) .. (205.1,184.2) -- cycle ;
\draw  [fill={rgb, 255:red, 80; green, 227; blue, 194 }  ,fill opacity=0.5 ] (511,161.1) .. controls (511,154.97) and (515.97,150) .. (522.1,150) -- (555.4,150) .. controls (561.53,150) and (566.5,154.97) .. (566.5,161.1) -- (566.5,260.4) .. controls (566.5,266.53) and (561.53,271.5) .. (555.4,271.5) -- (522.1,271.5) .. controls (515.97,271.5) and (511,266.53) .. (511,260.4) -- cycle ;
\draw  [fill={rgb, 255:red, 80; green, 227; blue, 194 }  ,fill opacity=0.5 ] (605,161.1) .. controls (605,154.97) and (609.97,150) .. (616.1,150) -- (649.4,150) .. controls (655.53,150) and (660.5,154.97) .. (660.5,161.1) -- (660.5,260.4) .. controls (660.5,266.53) and (655.53,271.5) .. (649.4,271.5) -- (616.1,271.5) .. controls (609.97,271.5) and (605,266.53) .. (605,260.4) -- cycle ;
\draw [line width=1.5]  [dash pattern={on 1.69pt off 2.76pt}]  (484.67,216) -- (499.5,216) ;
\draw [line width=1.5]  [dash pattern={on 1.69pt off 2.76pt}]  (577.67,216) -- (592.5,216) ;
\draw  [fill={rgb, 255:red, 0; green, 0; blue, 0 }  ,fill opacity=1 ] (430.5,180.8) .. controls (430.5,180.14) and (431.04,179.6) .. (431.7,179.6) .. controls (432.36,179.6) and (432.9,180.14) .. (432.9,180.8) .. controls (432.9,181.46) and (432.36,182) .. (431.7,182) .. controls (431.04,182) and (430.5,181.46) .. (430.5,180.8) -- cycle ;
\draw  [fill={rgb, 255:red, 0; green, 0; blue, 0 }  ,fill opacity=1 ] (430.5,209.8) .. controls (430.5,209.14) and (431.04,208.6) .. (431.7,208.6) .. controls (432.36,208.6) and (432.9,209.14) .. (432.9,209.8) .. controls (432.9,210.46) and (432.36,211) .. (431.7,211) .. controls (431.04,211) and (430.5,210.46) .. (430.5,209.8) -- cycle ;
\draw  [fill={rgb, 255:red, 0; green, 0; blue, 0 }  ,fill opacity=1 ] (430.5,237.8) .. controls (430.5,237.14) and (431.04,236.6) .. (431.7,236.6) .. controls (432.36,236.6) and (432.9,237.14) .. (432.9,237.8) .. controls (432.9,238.46) and (432.36,239) .. (431.7,239) .. controls (431.04,239) and (430.5,238.46) .. (430.5,237.8) -- cycle ;
\draw  [fill={rgb, 255:red, 0; green, 0; blue, 0 }  ,fill opacity=1 ] (456.5,198.8) .. controls (456.5,198.14) and (457.04,197.6) .. (457.7,197.6) .. controls (458.36,197.6) and (458.9,198.14) .. (458.9,198.8) .. controls (458.9,199.46) and (458.36,200) .. (457.7,200) .. controls (457.04,200) and (456.5,199.46) .. (456.5,198.8) -- cycle ;
\draw  [fill={rgb, 255:red, 0; green, 0; blue, 0 }  ,fill opacity=1 ] (457,226.3) .. controls (457,225.64) and (457.54,225.1) .. (458.2,225.1) .. controls (458.86,225.1) and (459.4,225.64) .. (459.4,226.3) .. controls (459.4,226.96) and (458.86,227.5) .. (458.2,227.5) .. controls (457.54,227.5) and (457,226.96) .. (457,226.3) -- cycle ;
\draw [color={rgb, 255:red, 144; green, 19; blue, 254 }  ,draw opacity=1 ]   (525.2,181.3) -- (551.2,199.3) ;
\draw [color={rgb, 255:red, 144; green, 19; blue, 254 }  ,draw opacity=1 ]   (525.7,210) -- (551.7,226.8) ;
\draw [color={rgb, 255:red, 144; green, 19; blue, 254 }  ,draw opacity=1 ]   (525.2,210.3) -- (551.2,199.5) ;
\draw [color={rgb, 255:red, 144; green, 19; blue, 254 }  ,draw opacity=1 ]   (525.2,238.3) -- (551.7,226.8) ;
\draw [color={rgb, 255:red, 144; green, 19; blue, 254 }  ,draw opacity=1 ]   (525.2,181.3) -- (551.7,225.6) ;
\draw [color={rgb, 255:red, 144; green, 19; blue, 254 }  ,draw opacity=1 ]   (525.2,238.3) -- (551.2,199.3) ;
\draw  [fill={rgb, 255:red, 0; green, 0; blue, 0 }  ,fill opacity=1 ] (524,181.3) .. controls (524,180.64) and (524.54,180.1) .. (525.2,180.1) .. controls (525.86,180.1) and (526.4,180.64) .. (526.4,181.3) .. controls (526.4,181.96) and (525.86,182.5) .. (525.2,182.5) .. controls (524.54,182.5) and (524,181.96) .. (524,181.3) -- cycle ;
\draw  [fill={rgb, 255:red, 0; green, 0; blue, 0 }  ,fill opacity=1 ] (524,210.3) .. controls (524,209.64) and (524.54,209.1) .. (525.2,209.1) .. controls (525.86,209.1) and (526.4,209.64) .. (526.4,210.3) .. controls (526.4,210.96) and (525.86,211.5) .. (525.2,211.5) .. controls (524.54,211.5) and (524,210.96) .. (524,210.3) -- cycle ;
\draw  [fill={rgb, 255:red, 0; green, 0; blue, 0 }  ,fill opacity=1 ] (524,238.3) .. controls (524,237.64) and (524.54,237.1) .. (525.2,237.1) .. controls (525.86,237.1) and (526.4,237.64) .. (526.4,238.3) .. controls (526.4,238.96) and (525.86,239.5) .. (525.2,239.5) .. controls (524.54,239.5) and (524,238.96) .. (524,238.3) -- cycle ;
\draw  [fill={rgb, 255:red, 0; green, 0; blue, 0 }  ,fill opacity=1 ] (550,199.3) .. controls (550,198.64) and (550.54,198.1) .. (551.2,198.1) .. controls (551.86,198.1) and (552.4,198.64) .. (552.4,199.3) .. controls (552.4,199.96) and (551.86,200.5) .. (551.2,200.5) .. controls (550.54,200.5) and (550,199.96) .. (550,199.3) -- cycle ;
\draw  [fill={rgb, 255:red, 0; green, 0; blue, 0 }  ,fill opacity=1 ] (550.5,226.8) .. controls (550.5,226.14) and (551.04,225.6) .. (551.7,225.6) .. controls (552.36,225.6) and (552.9,226.14) .. (552.9,226.8) .. controls (552.9,227.46) and (552.36,228) .. (551.7,228) .. controls (551.04,228) and (550.5,227.46) .. (550.5,226.8) -- cycle ;
\draw [color={rgb, 255:red, 144; green, 19; blue, 254 }  ,draw opacity=1 ]   (642.7,181.8) -- (616.7,199.8) ;
\draw [color={rgb, 255:red, 144; green, 19; blue, 254 }  ,draw opacity=1 ]   (642.2,210.5) -- (616.2,227.3) ;
\draw [color={rgb, 255:red, 144; green, 19; blue, 254 }  ,draw opacity=1 ]   (642.7,210.8) -- (616.7,200) ;
\draw [color={rgb, 255:red, 144; green, 19; blue, 254 }  ,draw opacity=1 ]   (642.7,238.8) -- (616.2,227.3) ;
\draw  [fill={rgb, 255:red, 0; green, 0; blue, 0 }  ,fill opacity=1 ] (643.9,181.8) .. controls (643.9,181.14) and (643.36,180.6) .. (642.7,180.6) .. controls (642.04,180.6) and (641.5,181.14) .. (641.5,181.8) .. controls (641.5,182.46) and (642.04,183) .. (642.7,183) .. controls (643.36,183) and (643.9,182.46) .. (643.9,181.8) -- cycle ;
\draw  [fill={rgb, 255:red, 0; green, 0; blue, 0 }  ,fill opacity=1 ] (643.9,210.8) .. controls (643.9,210.14) and (643.36,209.6) .. (642.7,209.6) .. controls (642.04,209.6) and (641.5,210.14) .. (641.5,210.8) .. controls (641.5,211.46) and (642.04,212) .. (642.7,212) .. controls (643.36,212) and (643.9,211.46) .. (643.9,210.8) -- cycle ;
\draw  [fill={rgb, 255:red, 0; green, 0; blue, 0 }  ,fill opacity=1 ] (643.9,238.8) .. controls (643.9,238.14) and (643.36,237.6) .. (642.7,237.6) .. controls (642.04,237.6) and (641.5,238.14) .. (641.5,238.8) .. controls (641.5,239.46) and (642.04,240) .. (642.7,240) .. controls (643.36,240) and (643.9,239.46) .. (643.9,238.8) -- cycle ;
\draw  [fill={rgb, 255:red, 0; green, 0; blue, 0 }  ,fill opacity=1 ] (617.9,199.8) .. controls (617.9,199.14) and (617.36,198.6) .. (616.7,198.6) .. controls (616.04,198.6) and (615.5,199.14) .. (615.5,199.8) .. controls (615.5,200.46) and (616.04,201) .. (616.7,201) .. controls (617.36,201) and (617.9,200.46) .. (617.9,199.8) -- cycle ;
\draw  [fill={rgb, 255:red, 0; green, 0; blue, 0 }  ,fill opacity=1 ] (617.4,227.3) .. controls (617.4,226.64) and (616.86,226.1) .. (616.2,226.1) .. controls (615.54,226.1) and (615,226.64) .. (615,227.3) .. controls (615,227.96) and (615.54,228.5) .. (616.2,228.5) .. controls (616.86,228.5) and (617.4,227.96) .. (617.4,227.3) -- cycle ;

\draw (56.5,282) node [anchor=north west][inner sep=0.75pt]  [font=\normalsize] [align=left] {$V_1$};
\draw (127,282) node [anchor=north west][inner sep=0.75pt]  [font=\normalsize] [align=left] {$V_2$};
\draw (197,282) node [anchor=north west][inner sep=0.75pt]  [font=\normalsize] [align=left] {$V_3$};
\draw (294.5,282) node [anchor=north west][inner sep=0.75pt]  [font=\normalsize] [align=left] {$V_j$};
\draw (85,215) node [anchor=north west][inner sep=0.75pt]  [font=\scriptsize] [align=left] {$\phi_v(1)$};
\draw (155,243) node [anchor=north west][inner sep=0.75pt]  [font=\scriptsize] [align=left] {$\phi_v(2)$};
\draw (225,243) node [anchor=north west][inner sep=0.75pt]  [font=\scriptsize] [align=left] {$\phi_v(3)$};
\draw (320,167.67) node [anchor=north west][inner sep=0.75pt]  [font=\scriptsize] [align=left] {$\phi_v(j)$};
\draw (301.7,186.6) node [anchor=north west][inner sep=0.75pt]  [font=\scriptsize] [align=left] {$v$};
\draw (435,282) node [anchor=north west][inner sep=0.75pt]  [font=\normalsize] [align=left] {$F_{1,c}$};
\draw (530,282) node [anchor=north west][inner sep=0.75pt]  [font=\normalsize] [align=left] {$F_{i,c}$};
\draw (622,282) node [anchor=north west][inner sep=0.75pt]  [font=\normalsize] [align=left] {$F_{m,c}$};
\draw (40,310) node [anchor=north west][inner sep=0.75pt]   [align=left] {$\phi_v(1),\phi_v(2),\phi_v(3),\ldots,\phi_v(j)$ are distinct};
\draw (460,310) node [anchor=north west][inner sep=0.75pt]   [align=left] {$
F_{1,c}\mathbin{\dot\cup}F_{2,c}
\mathbin{\dot\cup}\cdots
\mathbin{\dot\cup}F_{m,c}$};

\end{tikzpicture}
\caption{The colouring in Step~1. Left: the base colour associated with
an edge from $V_i$ to $v$ is $\phi_v(i)$. Right: the colour-$c$ graph
is the vertex-disjoint union of $F_{1,c},\ldots,F_{m,c}$.}\label{fig:construction-overview}
\end{figure}
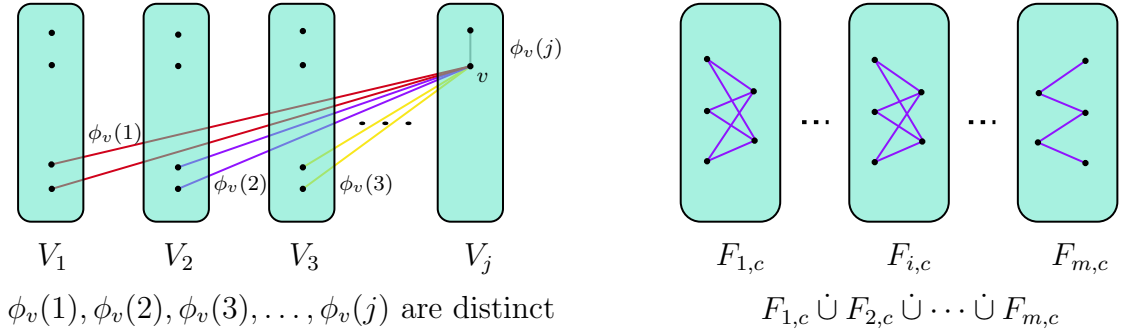

\section{The complete-graph case}\label{sec:complete}

\subsection{Construction and basic properties}
Let $q$ be a sufficiently large integer, and regard $[q]$ as the set of
base colours. Asymptotic notation refers to the
limit $q\to\infty$. Set
$
m=\left\lfloor{q^{1/5}}/{\log q}\right\rfloor.
$ Choose an integer $T$ between $ qm^3(\log m)^2$ and $2 qm^3(\log m)^2$.
Let $V_1,\ldots,V_m$ be pairwise disjoint sets of size $T$. Put $V=V_1\mathbin{\dot\cup}\cdots\mathbin{\dot\cup}V_m$, let $K_V$
denote the complete graph on $V$, and set $n=|V|=mT$.

The construction can be viewed roughly as follows. For each colour
$c\in[q]$, the colour-$c$ subgraph is the union
$F_{1,c}\cup\cdots\cup F_{m,c}$. The graph $F_{i,c}$ is  bipartite, with one side almost spanning
$V_i$ and the other contained in $V_i\cup\cdots\cup V_m$.
For each $a>i$, its restriction to $(V_i,V_a)$ is roughly a complete
bipartite graph with parts of sizes $(1-o(1))T$ and $T/q$. We also introduce a sparse graph $J$, whose edges are coloured with
$o(q)$ new colours, called auxiliary colours. Its purpose is to
separate these pieces: for every $c\in[q]$, the graphs
$F_{1,c},\ldots,F_{m,c}$ are pairwise vertex-disjoint.

Fix a total order $\prec$ on $V$ such that, for every $i<j$, every
$u\in V_i$, and every $v\in V_j$, we have $u\prec v$. Within each class,
the vertices are ordered arbitrarily.  For each vertex, we associate a distinct colour with each of the $m$
classes. To this end, independently for each $v\in V$, choose a uniformly
random injection $\phi_v:[m]\hookrightarrow[q]$. Thus $\phi_v(i)$ is the colour associated with the pair $(v,V_i)$ and will be the colour of almost all edges $uv$ with $u\in V_i$ and $u\prec v$. We call $\phi_v$ the label of $v$.
For
each $v\in V$, let $\operatorname{part}(v)$ denote the unique $i\in[m]$
such that $v\in V_i$.

\begin{construction}\label{con:colouring}
We construct an edge-colouring of $K_V$ in three
steps.

\begin{enumerate}
\item Define a graph $J$ with vertex set $V$. For each pair $x\prec y$, include
the edge $xy$ in $J$ if
$\phi_y\bigl(\operatorname{part}(x)\bigr)\in\im\phi_x$; that is, if the colour associated with the pair
$(y,V_{\operatorname{part}(x)})$ is also assigned by $x$ to one of the
classes.

\item For each edge $xy$ with $x\prec y$ and $xy\notin E(J)$, we classify it according to
$\bigl(\operatorname{part}(x),\phi_y(\operatorname{part}(x))\bigr)$.
For each $i\in[m]$ and $c\in[q]$, let
$$
E(F_{i,c})
=
\left\{
xy\in\binom{V}{2}:
x\prec y,\ x\in V_i,\ xy\notin E(J),\ \phi_y(i)=c
\right\},
$$
and let $V(F_{i,c})$ consist of the endpoints of these edges.

\item Colour every edge of $F_{i,c}$ with the base colour $c$. Properly
edge-colour $J$ with at most $\Delta(J)+1$ fresh auxiliary colours,
using Vizing's theorem.
\end{enumerate}
\end{construction}

By definition, the edges of the graphs $F_{i,c}$ partition
$E(K_V)\setminus E(J)$, so the construction gives an edge-colouring of
$K_V$.
The following lemma records the main properties
of the construction.

\begin{lemma}\label{lem:construction-properties}
With high probability, the colouring in Construction~\ref{con:colouring}
has the following properties.
\begin{enumerate}[label=\textup{(\roman*)}]

\item\label{item:J-degree}
Every vertex $v\in V$ satisfies
$
d_J(v)=(1+o(1))m^2T/q.
$
In particular,
$
\Delta(J)=O(q/(\log q)^3),
$
and the construction uses $(1+o(1))q$ colours in total.

\item\label{item:F-bipartite}
Simultaneously for all $i\in[m]$ and $c\in[q]$, the graph $F_{i,c}$ is
bipartite. Moreover, for every $a>i$, the restriction of $F_{i,c}$ to
$(V_i,V_a)$ is a complete bipartite graph whose parts have sizes
$(1+o(1))T$ and $(1+o(1))T/q$, respectively.

\item\label{item:auxiliary-matching}
Every auxiliary colour class is a matching.

\item\label{item:F-disjoint}
For each $c\in[q]$, the graphs $F_{1,c},\ldots,F_{m,c}$ are pairwise
vertex-disjoint. Consequently, every connected monochromatic subgraph
on at least two vertices of colour $c$ lies in a unique $F_{i,c}$.
\end{enumerate}
\end{lemma}
\begin{proof}
For \ref{item:J-degree}, fix $x\in V$ and condition on $\phi_x$. For
each $y\ne x$, the event $xy\in E(J)$ has probability $m/q$, and these
events are conditionally independent. Hence
$d_J(x)\sim\operatorname{Bin}(n-1,m/q)$, with mean
$\mu=(n-1)m/q=(1+o(1))m^2T/q$. Since $\mu\gg\log n$, Chernoff's
inequality and a union bound give
$d_J(x)=(1+o(1))m^2T/q$ simultaneously for all $x\in V$ with high
probability. In particular,
$\Delta(J)=O(m^2T/q)=O(q/(\log q)^3)$. By Vizing's theorem, the number of auxiliary colours is therefore
$o(q)$, so the construction uses $(1+o(1))q$ colours in total.

For \ref{item:F-bipartite}, an edge $xy\in F_{i,c}$ with $x\prec y$
satisfies $x\in V_i$, $c\notin\im\phi_x$, and $\phi_y(i)=c$. Hence,
noting that isolated vertices are not included in $F_{i,c}$ by definition,
the graph $F_{i,c}$ is bipartite with parts contained in
$\{x\in V_i:c\notin\im\phi_x\}$ and
$\{y\in V_i\cup\cdots\cup V_m:\phi_y(i)=c\}$. These two sets are disjoint, since $\phi_y(i)=c$ implies
$c\in\operatorname{im}\phi_y$. For $a>i$, its restriction to $(V_i,V_a)$ is
the complete bipartite graph between
$\{x\in V_i:c\notin\im\phi_x\}$ and
$\{y\in V_a:\phi_y(i)=c\}$. The two sets have expected sizes
$(1-m/q)T=(1-o(1))T$ and $T/q$, respectively. Since $T/q\gg\log(mq)$,
Chernoff's inequality and a union bound over all $i,c,a$ give the
claimed sizes simultaneously with high probability.

Property~\ref{item:auxiliary-matching} follows immediately from the
proper edge-colouring of $J$.

For \ref{item:F-disjoint}, fix $c\in[q]$. Suppose for a contradiction
that $v\in V(F_{i,c})\cap V(F_{j,c})$ for distinct $i,j$. If $v$ is an earlier endpoint of an edge in $F_{i,c}$, then $v\in V_i$,
and since this edge lies outside $J$, we have $c\notin\im\phi_v$.
Hence $v$ cannot be a later endpoint of an edge in $F_{j,c}$, since
that would require $\phi_v(j)=c$. Moreover, since $i\ne j$, we have
$v\notin V_j$, so $v$ cannot be an earlier endpoint in $F_{j,c}$
either. Thus $v$ cannot be an earlier endpoint in either graph.
Consequently, it must be a later endpoint in both, so
$\phi_v(i)=\phi_v(j)=c$, contradicting the injectivity. 
\end{proof}
For a fixed $v\in V$ and a colour set $D_i\subseteq[q]$, the condition
$\phi_v(i)\notin D_i$ means that the colour associated with the pair
$(v,V_i)$ avoids $D_i$. In particular, $v$ cannot be the later endpoint
of an edge in $F_{i,c}$ for any $c\in D_i$. Since the labels $\phi_v$, $v\in V$, are identically distributed,
let $\Phi:[m]\hookrightarrow[q]$ denote a uniformly random injection with this
common distribution. 
We record two properties of the labels. The sets $D_i$ below are not required to be distinct for
different $i$.

\begin{lemma}\label{lem:label-properties}
The following properties hold.
\begin{enumerate}[label=\textup{(\roman*)}]
\item\label{item:avoidance}
For every $I\subseteq[m]$ and family $(D_i)_{i\in I}$ of subsets
of $[q]$ satisfying $|D_i|\le q-m$ for all $i\in I$,
$$
\Pr\bigl(\Phi(i)\notin D_i\text{ for every }i\in I\bigr)
\ge
\prod_{i\in I}\left(1-\frac{|D_i|}{q-m}\right).
$$

\item\label{item:uniformity2}
For $a\in[m]$, $I\subseteq[m]$, and a family
$\mathcal D=(D_i)_{i\in I}$ of subsets of $[q]$, let
$Z(a,I,\mathcal D)$ be the number of vertices $u\in V_a$ such that
$\phi_u(i)\notin D_i$ for every $i\in I$. Then, with high probability,
simultaneously for all such choices,
$$
\frac{Z(a,I,\mathcal D)}{T}
\ge
\Pr\bigl(\Phi(i)\notin D_i\text{ for every }i\in I\bigr)
-\frac{1}{m\log m}.
$$
\end{enumerate}
\end{lemma}

\begin{proof}
For \ref{item:avoidance}, expose the values $\Phi(i)$, $i\in I$,
sequentially. At each step, at least $q-m$ colours remain available, so
the conditional probability of avoiding $D_i$ is at least
$1-|D_i|/(q-m)$. Multiplying these bounds over $i\in I$ proves
\ref{item:avoidance}.

For \ref{item:uniformity2}, since the labels $\phi_u$, $u\in V_a$,
are independent and identically distributed, $Z(a,I,\mathcal D)$ is
the sum of $T$ independent Bernoulli random variables with common mean
$\Pr\bigl(\Phi(i)\notin D_i\text{ for every }i\in I\bigr)$. Hence
Hoeffding's inequality gives
$$
\Pr\left(
\frac{Z(a,I,\mathcal D)}{T}
<
\Pr\bigl(\Phi(i)\notin D_i\text{ for every }i\in I\bigr)
-\frac{1}{m\log m}
\right)
\le
\exp\left(-\frac{2T}{m^2(\log m)^2}\right).
$$
There are at most $m(1+2^q)^m<e^{qm}$ choices of
$(a,I,\mathcal D)$, while $T/(m^2(\log m)^2)\ge qm$. A union bound
therefore proves \ref{item:uniformity2}.
\end{proof}

\subsection{Matchings in the construction and the proof of Theorem~\ref{thm:main}}
Fix a realization of the labels for which
Lemma~\ref{lem:construction-properties}
and Lemma~\ref{lem:label-properties} both hold.
We now analyse matchings in the construction. For a matching $M$, let
$A(M)=|M\cap E(J)|$, and for each $i\in[m]$ let
$
C_i(M)=\{c\in[q]:M\cap E(F_{i,c})\ne\varnothing\}.
$
Thus $A(M)$ counts the auxiliary-coloured edges of $M$, while $C_i(M)$ records the base colours used by $M$ in the graphs
$F_{i,c}$, and $\sum_{i=1}^m|C_i(M)|$ is the number of pairs
$(i,c)$ for which $M$ uses an edge of $F_{i,c}$.

\begin{lemma}\label{lem:path-matching}
Suppose that Construction~\ref{con:colouring} is partitioned into $\ell$
monochromatic paths. Then there is a matching $M$ such that
$
A(M)+\sum_{i=1}^m|C_i(M)|\le\ell\text{ and } 
n-2e(M)\le\ell.
$
\end{lemma}

\begin{proof}
Choose a maximum matching in each path, and let $M$ be their union.
By Lemma~\ref{lem:construction-properties}\ref{item:auxiliary-matching},
a nontrivial path of an auxiliary colour consists of a single edge and hence
contributes one to $A(M)$. By Lemma~\ref{lem:construction-properties}\ref{item:F-disjoint}, any
path of base colour $c$ and length at least $1$ lies in $F_{i,c}$ for a
unique $i\in[m]$ and contributes at most one to
$\sum_{i=1}^m|C_i(M)|$. This proves the first inequality. The second
follows because a maximum matching in a path leaves at most one vertex
unmatched.
\end{proof}

Our goal is to show that every monochromatic path partition of the
constructed graph contains at least $(1-o(1))q\log\log m$ paths. Thus, given a
partition into $\ell$ monochromatic paths, we may assume that
$\ell<q\log\log m$. Let $M$ be the matching supplied by
Lemma~\ref{lem:path-matching}, and set
$
A=A(M), D=n-2e(M),
X=\frac1q\sum_{i=1}^m|C_i(M)|.
$
Then
$$
A+qX\le\ell<q\log\log m
\qquad\text{and}\qquad
D\le\ell<q\log\log m.
$$
In particular,
$$
X<\log\log m
\qquad\text{and}\qquad
A+D<2q\log\log m.
$$
The latter means that $M$ covers almost all vertices and uses only a few
auxiliary-coloured edges, while the former means that its base-coloured
edges are supported on relatively few pairs $(i,c)$. The next lemma shows that $X$ cannot be too small. Intuitively, if $X$ were too small, then the total size of the sets
$C_i(M)$ would be small. Since the labels are random, many vertices
would avoid $C_i(M)$ for most earlier indices $i$ and hence have very
few ways to be matched. As $M$ uses only a few auxiliary edges and leaves only a
few vertices unmatched, this would be impossible.
\begin{lemma}\label{lem:matching-support}
Let $M$ be a matching in Construction~\ref{con:colouring}, and set
$A=A(M)$, $D=n-2e(M)$, and
$X=\frac1q\sum_{i=1}^m|C_i(M)|$.
If $X<\log\log m$ and $A+D<2q\log\log m$, then
$X\ge(1-o(1))\log\log m$, where the $o(1)$ tends to zero as $q\to\infty$ and is independent of the choice of $M$.
\end{lemma}

This lemma follows from the following technical result, which shows
that there is a choice of $s$, corresponding
to the final $s$ classes $V_{m-s+1},\ldots,V_m$, such that, after
excluding the set $L$ of preceding classes on which $M$ uses more than
$q/\log m$ base colours, the probability that $\Phi(i)\notin C_i(M)$
for every remaining preceding class $i$ is small. This event is relevant because
Lemma~\ref{lem:label-properties}\ref{item:uniformity2} translates its
probability into a lower bound on the proportion of vertices whose
labels satisfy the corresponding avoidance conditions. Note that if a vertex $v$ satisfies
$\phi_v(i)\notin C_i(M)$, then $v$ cannot be the later endpoint of a
base-coloured edge of $M$ whose earlier endpoint lies in $V_i$, since
every such edge with later endpoint $v$ belongs to
$F_{i,\phi_v(i)}$. The upper bound on this avoidance probability given
by the next lemma will then be compared with the lower bound from
Lemma~\ref{lem:label-properties}\ref{item:avoidance}, forcing $X$ to be
large.

\begin{lemma}\label{lem:tail-covering}
Let $M$ be a matching in Construction~\ref{con:colouring}, and set
$A=A(M)$, $D=n-2e(M)$, and
$X=\frac1q\sum_{i=1}^m|C_i(M)|$.
Suppose that $A+D\le2q\log\log m$. Then there exists an integer $s$ with
$(\log m)^3\le s\le m/2$ such that, for
$L=\{i\in[m-s]:|C_i(M)|>q/\log m\}$,
$$
\Pr\bigl(\Phi(i)\notin C_i(M)\text{ for every }i\in[m-s]\setminus L\bigr)
=O\left(\frac{X+1}{\log m}\right).
$$
\end{lemma}

\begin{proof}
Let $s$ be any integer with $(\log m)^3\le s\le m/2$, and let $\rho$
denote the probability on the left-hand side of the desired inequality. Let $U$
be the set of vertices $u\in V_{m-s+1}\cup\cdots\cup V_m$ such that
$\phi_u(i)\notin C_i(M)$ for every $i\in[m-s]\setminus L$.

We bound $\rho$ through the size of $U$. On the one hand, by
Lemma~\ref{lem:label-properties}\ref{item:uniformity2}, the probability
$\rho$ is reflected in each class $V_a$, where $m-s+1\le a\le m$, giving
a lower bound on $|U|$. On the other hand, vertices
of $U$ are difficult to match using $M$. By the definition of $U$, we
will show that no vertex of $U$ can be matched by a base-coloured edge
to $\bigcup_{i\in[m-s]\setminus L}V_i$. Thus the vertices of $U$ are
accounted for by those matched into $\bigcup_{i\in L}V_i$, those covered
by edges of $M\cap E(J)$, those covered by edges of $M$ whose earlier
endpoints lie in $\bigcup_{i=m-s+1}^m V_i$, and those left unmatched.
Comparing the resulting upper bound on $|U|$ with the lower bound above
gives an upper bound on $\rho$.

\smallskip
\noindent\textbf{Claim 1.}
We have $|U|\ge sT(\rho-1/(m\log m))$, and at most $|L|T$
vertices of $U$ are matched by $M$ to $V_1\cup\cdots\cup V_{m-s}$
through base-coloured edges.

\begin{proof}
For each $a\in\{m-s+1,\ldots,m\}$, apply
Lemma~\ref{lem:label-properties}\ref{item:uniformity2} with
$I=[m-s]\setminus L$ and $D_i=C_i(M)$. We obtain
$|U\cap V_a|\ge T(\rho-1/(m\log m))$. Summing over these $s$ values of $a$
gives
$
|U|\ge sT(\rho-1/(m\log m)).
$

Suppose that a base-coloured edge of $M$ has earlier endpoint in $V_i$,
where $i\le m-s$, and later endpoint $u\in U$. Since its colour lies in $C_i(M)$, while the definition of $U$ gives
$\phi_u(i)\notin C_i(M)$ whenever $i\notin L$, we must have $i\in L$. Thus every base-coloured edge of $M$ matching a vertex of $U$ to
$V_1\cup\cdots\cup V_{m-s}$ has its other endpoint in
$\bigcup_{i\in L}V_i$, which has size $|L|T$.
\end{proof}

We then bound the number of base-coloured matching edges starting in
the final $s$ classes.

\smallskip
\noindent\textbf{Claim 2.}
The number of base-coloured edges of $M$ whose earlier
endpoints lie in $V_{m-s+1}\cup\cdots\cup V_m$ is at most
$$
\frac{T}{q}\sum_{i=m-s+1}^m(m-i+1)|C_i(M)|
+\frac{Ts(s+1)}{2m\log m}.
$$

\begin{proof}
Fix $i\in\{m-s+1,\ldots,m\}$. If such an edge has earlier endpoint in
$V_i$ and later endpoint $u\in V_a$, then $a\ge i$ and
$\phi_u(i)\in C_i(M)$. Applying Lemma~\ref{lem:label-properties}\ref{item:uniformity2} with
$I=\{i\}$ and $D_i=C_i(M)$, and using
$\Pr(\Phi(i)\notin C_i(M))=1-|C_i(M)|/q$, there are at most
$T(|C_i(M)|/q+1/(m\log m))$ possible later endpoints in each $V_a$. 
Thus, for fixed $i$, there are at most
$(m-i+1)T(|C_i(M)|/q+1/(m\log m))$ such edges. Summing over
$i=m-s+1,\ldots,m$ proves the claim.
\end{proof}

By Claim~1,
$|U|\ge sT(\rho-1/(m\log m))$. On the other hand, at most $|L|T$
vertices of $U$ are matched by base-coloured edges to the preceding
classes, at most $2A$ are covered by edges of $J$, at most twice the
bound in Claim~2 are covered by base-coloured edges whose earlier
endpoints lie in the final $s$ classes, and at most $D$ are unmatched.
Hence
$$
sT\left(\rho-\frac{1}{m\log m}\right)
\le
|L|T+2A
+\frac{2T}{q}\sum_{i=m-s+1}^m(m-i+1)|C_i(M)|
+\frac{Ts(s+1)}{m\log m}+D.
$$
Dividing by $sT$ gives
$$
\rho\le
\frac{|L|}{s}
+\frac{2}{qs}\sum_{i=m-s+1}^m(m-i+1)|C_i(M)|
+\frac{s+2}{m\log m}
+\frac{2A+D}{sT}.
$$
Since $s\le m/2$ and $A+D\le2q\log\log m$, the last two terms are
$O(1/\log m)$ and $o(1/\log m)$, respectively. We now choose $s$ so
that the first two terms are small.

\smallskip
\noindent\textbf{Claim 3.}
There exists an integer $s$ with $(\log m)^3\le s\le m/2$ such that
$$
\frac{1}{qs}\sum_{i=m-s+1}^m(m-i+1)|C_i(M)|
=O\left(\frac{X}{\log m}\right)
\quad\text{and}\quad
\frac{|L|}{s}=O\left(\frac{X}{(\log m)^2}\right).
$$

\begin{proof}
Let $\mathcal S$ be the set of powers of two in
$[(\log m)^3,m/2]$. Changing the order of summation gives
$$
\begin{aligned}
\sum_{s\in\mathcal S}\frac{1}{qs}
\sum_{i=m-s+1}^m(m-i+1)|C_i(M)|
&=
\frac{1}{q}\sum_{i=1}^m|C_i(M)|
\sum_{\substack{s\in\mathcal S\\ s\ge m-i+1}}
\frac{m-i+1}{s} \\
&\le \frac{2}{q}\sum_{i=1}^m|C_i(M)|
\\&=2X.
\end{aligned}
$$
Since $|\mathcal S|=\Theta(\log m)$, some $s\in\mathcal S$ satisfies
$$
\frac{1}{qs}\sum_{i=m-s+1}^m(m-i+1)|C_i(M)|
=O\left(\frac{X}{\log m}\right).
$$
For this $s$, the definition of $L$ gives
$q|L|/\log m<\sum_{i\in L}|C_i(M)|\le qX$, and hence
$|L|\le X\log m$. Since $s\ge(\log m)^3$, it follows that
$|L|/s=O(X/(\log m)^2)$.
\end{proof}
Choose $s$ as in Claim~3. Then the first two terms in the preceding
bound for $\rho$ are $O(X/\log m)$. Therefore
$\rho=O((X+1)/\log m)$, as required.
\end{proof}

\begin{proof}[Proof of Lemma~\ref{lem:matching-support}]
Choose $s$ and $L$ as in Lemma~\ref{lem:tail-covering}, and let $\rho$
denote the probability appearing in that lemma. Then
$\rho=O((X+1)/\log m)$.

Since $|C_i(M)|\le q/\log m\le q-m$ for every
$i\in[m-s]\setminus L$ and $m=o(q)$, we have
$|C_i(M)|/(q-m)=o(1)$ uniformly over $i\in[m-s]\setminus L$.
Thus, applying
Lemma~\ref{lem:label-properties}\ref{item:avoidance} with
$I=[m-s]\setminus L$ and $D_i=C_i(M)$, and using
$\log(1-x)\ge -x-2x^2$ for sufficiently small $x\ge0$, we obtain
$$
\begin{aligned}
\rho
&\ge
\prod_{i\in[m-s]\setminus L}
\left(1-\frac{|C_i(M)|}{q-m}\right) \\
&\ge
\exp\left(
-\sum_{i\in[m-s]\setminus L}\frac{|C_i(M)|}{q-m}
-2\sum_{i\in[m-s]\setminus L}
\left(\frac{|C_i(M)|}{q-m}\right)^2
\right) \\
&\ge
\exp\left(
-\frac{qX}{q-m}
-\frac{2q}{(q-m)\log m}
\sum_{i\in[m-s]\setminus L}\frac{|C_i(M)|}{q-m}
\right) \\
&\ge
\exp\left(
-\frac{qX}{q-m}
-\frac{2q^2X}{(q-m)^2\log m}
\right)
\\ &=e^{-X-o(1)},
\end{aligned}
$$
where the last equality uses $m=o(q)$ and $X<\log\log m$. Comparing this with the upper bound on $\rho$ gives
$X+\log(X+1)\ge\log\log m-O(1)$. Since $X<\log\log m$, we have
$\log(X+1)=o(\log\log m)$, and hence
$X\ge(1-o(1))\log\log m$.
\end{proof}
\begin{proof}[Proof of Theorem~\ref{thm:main}]
Let $r$ be any sufficiently large integer and take
$q=\lfloor r-r/(\log r)^2\rfloor=(1-o(1))r$. We consider
Construction~\ref{con:colouring} with this choice of $q$. By
Lemma~\ref{lem:construction-properties}\ref{item:J-degree}, the graph
$J$ can be properly edge-coloured using $O(q/(\log q)^3)\le r-q$
auxiliary colours. Hence the construction uses at most $r$ colours.
Consider a partition of Construction~\ref{con:colouring} into $\ell$
monochromatic paths. Let $M$ be the matching supplied by
Lemma~\ref{lem:path-matching}, and set $A=A(M)$, $D=n-2e(M)$, and
$X=\frac1q\sum_{i=1}^m|C_i(M)|$.

We claim that $\ell\ge(1-o(1))q\log\log m$. The claim is immediate if
$\ell\ge q\log\log m$, so suppose that $\ell<q\log\log m$. Then $X<\log\log m$ and $A+D<2q\log\log m$, so
Lemma~\ref{lem:matching-support} gives
$X\ge(1-o(1))\log\log m$.
Hence
$$
\ell\ge qX\ge(1-o(1))q\log\log m,
$$
as required.

Finally, $m=(1+o(1))q^{1/5}/\log q$, so
$\log\log m=(1-o(1))\log\log r$. Together with $q=(1-o(1))r$, this
yields $\ell\ge(1-o(1))r\log\log r$. By the definition of
$\ppath(r)$, we conclude that
$\ppath(r)\ge(1-o(1))r\log\log r$.
\end{proof}
\subsection{Blow-ups of the construction and the proof of Theorem~\ref{cor:arbitrarily-large}}
\label{sec:blowup}
We begin by defining the blow-up of Construction~\ref{con:colouring}.
\begin{construction}\label{con:blowup}
Fix a positive integer $h$ and start with Construction~\ref{con:colouring}.
Replace each vertex $v\in V$ by a cluster $B_v$ of size $h$. For
distinct $x,y\in V$, colour every edge between $B_x$ and $B_y$ with
the colour of $xy$ in the original construction, and, for each $v\in V$,
colour all edges inside the cluster $B_v$ with a single fresh auxiliary
colour, the same for every cluster.

For each $i\in[m]$, set
$\widetilde V_i=\bigcup_{v\in V_i}B_v$. Every vertex of $B_v$ inherits
the label $\phi_v$, and, for each $c\in[q]$, let
$\widetilde F_{i,c}$ be the blow-up of $F_{i,c}$ obtained by replacing
each vertex $v$ with $B_v$. Thus the base-coloured edges are precisely
the edges of the graphs $\widetilde F_{i,c}$.

Finally, if a base-coloured edge joins $B_x$ and $B_y$ with $x\prec y$,
we call its endpoint in $B_x$ the earlier endpoint and its endpoint in
$B_y$ the later endpoint.
\end{construction}

For a matching $M$ in the blow-up, let $A=A(M)$ be the number of
non-base-coloured edges of $M$, and let
$C_i(M)=\{c\in[q]:M\cap E(\widetilde F_{i,c})\ne\varnothing\}$.
Set $D=hn-2e(M)$ and
$X=\frac{1}{q}\sum_{i=1}^m|C_i(M)|$. Note that, unlike in the original construction, $A$ now counts all
non-base-coloured edges, including those lying inside the clusters.

\begin{lemma}\label{lem:blowup-matching}
For every positive integer $h$, if $X\le\log\log m$ and
$A+D\le2hq\log\log m$, then
$X\ge(1-o(1))\log\log m$. Here the $o(1)$ is with respect to $q\to\infty$, and its value does not
depend on $h$.
\end{lemma}

\begin{proof}
Write $C_i=C_i(M)$. Let $s$ be any integer with
$(\log m)^3\le s\le m/2$, set
$L=\{i\in[m-s]:|C_i|>q/\log m\}$, and let
$\rho=\Pr(\Phi(i)\notin C_i\text{ for every }i\in[m-s]\setminus L)$.
Let $U$ be the set of vertices
$u\in\widetilde V_{m-s+1}\cup\cdots\cup\widetilde V_m$ such that
$\phi_u(i)\notin C_i$ for every $i\in[m-s]\setminus L$.

As in the proof of Lemma~\ref{lem:tail-covering}, we bound $\rho$
through the size of $U$. The only difference is that every original
vertex is replaced by $h$ vertices with the same label.

\smallskip
\noindent\textbf{Claim 1.}
We have
$
|U|\ge shT\left(\rho-{1}/({m\log m})\right),
$
and at most $|L|hT$ vertices of $U$ are matched by
base-coloured edges of $M$ to
$\widetilde V_1\cup\cdots\cup\widetilde V_{m-s}$.

\begin{proof}
This is the same argument as in Claim~1 of
Lemma~\ref{lem:tail-covering}. Since every vertex of $V_a$ is replaced
by $h$ vertices with the same label,
Lemma~\ref{lem:label-properties}\ref{item:uniformity2} gives
$|U\cap\widetilde V_a|\ge hT(\rho-1/(m\log m))$ for each
$a\in\{m-s+1,\ldots,m\}$. Summing over these $s$ classes gives the
first assertion.

For the second, if a base-coloured edge has earlier endpoint in
$\widetilde V_i$ with $i\le m-s$ and later endpoint $u\in U$, then its
colour lies in $C_i$. Hence $i\in L$ by the definition of $U$.
The claim follows since $M$ is a matching.
\end{proof}

We next bound the number of base-coloured matching edges starting in
the final $s$ classes.

\smallskip
\noindent\textbf{Claim 2.}
The number of base-coloured edges of $M$ whose earlier endpoints lie in
$\widetilde V_{m-s+1}\cup\cdots\cup\widetilde V_m$ is at most
$$
\frac{hT}{q}\sum_{i=m-s+1}^m(m-i+1)|C_i|
+\frac{hTs(s+1)}{2m\log m}.
$$

\begin{proof}
This is the blow-up analogue of Claim~2 of
Lemma~\ref{lem:tail-covering}. For fixed
$i\in\{m-s+1,\ldots,m\}$ and $a\ge i$, every possible later endpoint
$u\in\widetilde V_a$ satisfies $\phi_u(i)\in C_i$. By the
one-coordinate case of
Lemma~\ref{lem:label-properties}\ref{item:uniformity2}, there are at
most $hT(|C_i|/q+1/(m\log m))$ such vertices. Summing over $a\ge i$
and then over $i$ proves the claim.
\end{proof}

By Claim~1, $|U|\ge shT(\rho-1/(m\log m))$. On the other hand, at most
$|L|hT$ vertices of $U$ are matched by base-coloured edges to the
preceding classes, at most $2A$ are covered by non-base-coloured edges,
at most twice the bound in Claim~2 are covered by base-coloured edges
whose earlier endpoints lie in the final $s$ classes, and at most $D$
are unmatched. Hence
$$
shT\left(\rho-\frac{1}{m\log m}\right)
\le
|L|hT+2A
+\frac{2hT}{q}\sum_{i=m-s+1}^m(m-i+1)|C_i|
+\frac{hTs(s+1)}{m\log m}+D.
$$
Dividing by $shT$ gives
$$
\rho\le
\frac{|L|}{s}
+\frac{2}{qs}\sum_{i=m-s+1}^m(m-i+1)|C_i|
+\frac{s+2}{m\log m}
+\frac{2A+D}{shT}.
$$
Since $s\le m/2$ and $A+D\le2hq\log\log m$, the last two terms are
$O(1/\log m)$ and $o(1/\log m)$, respectively. We now choose $s$ so
that the first two terms are small.

\smallskip
\noindent\textbf{Claim 3.}
There exists an integer $s$ with $(\log m)^3\le s\le m/2$ such that
$$
\frac{1}{qs}\sum_{i=m-s+1}^m(m-i+1)|C_i|
=O\left(\frac{X}{\log m}\right)
\quad\text{and}\quad
\frac{|L|}{s}
=O\left(\frac{X}{(\log m)^2}\right).
$$

\begin{proof}
This is exactly the same averaging argument as in Claim~3 of
Lemma~\ref{lem:tail-covering}.
\end{proof}

Choose $s$ as in Claim~3. Then the first two terms in the preceding
bound for $\rho$ are $O(X/\log m)$, and therefore
$
\rho=O\left({(X+1)}/{\log m}\right).
$

It remains to give the lower bound on $\rho$. This is the same argument
as in the proof of Lemma~\ref{lem:matching-support}. For every
$i\in[m-s]\setminus L$, we have $|C_i|\le q/\log m\le q-m$, so
Lemma~\ref{lem:label-properties}\ref{item:avoidance} gives
$$
\rho\ge
\prod_{i\in[m-s]\setminus L}
\left(1-\frac{|C_i|}{q-m}\right).
$$
As before,
$\sum_{i\in[m-s]\setminus L}|C_i|/(q-m)\le X+o(1)$ and
$\sum_{i\in[m-s]\setminus L}(|C_i|/(q-m))^2
=O(X/\log m)=o(1)$.
Hence $\rho\ge e^{-X-o(1)}$.

Comparing the two bounds for $\rho$ gives
$X+\log(X+1)\ge\log\log m-O(1)$. Since $X\le\log\log m$, we have
$\log(X+1)=o(\log\log m)$, and hence
$X\ge(1-o(1))\log\log m$.
\end{proof}

\begin{proof}[Proof of Theorem~\ref{cor:arbitrarily-large}]
Let $r$ be sufficiently large and set
$q=\lfloor r-r/(\log r)^2\rfloor$, as in the proof of
Theorem~\ref{thm:main}. Fix a realization of
Construction~\ref{con:colouring} satisfying
Lemmas~\ref{lem:construction-properties} and~\ref{lem:label-properties},
and consider Construction~\ref{con:blowup} for a positive integer $h$.

Suppose that the resulting coloured complete graph is partitioned into
$\ell$ monochromatic paths. Choose a maximum matching in each path and
let $M$ be the union of these matchings. A maximum matching in a path
leaves at most one vertex unmatched, so $D=hn-2e(M)\le\ell$.

Pairwise vertex-disjointness is preserved under blow-up. Hence, by
Lemma~\ref{lem:construction-properties}\ref{item:F-disjoint}, every
base-coloured path containing an edge lies in
$\widetilde F_{i,c}$ for a unique pair $(i,c)$. Consequently,
$qX=\sum_{i=1}^m|C_i(M)|\le\ell$.

For an original auxiliary colour, each connected component in the
blow-up is a copy of $K_{h,h}$, because every auxiliary colour class in
$J$ is a matching. The additional colour used inside the clusters has
components that are copies of $K_h$. Thus a maximum matching in any
non-base-coloured path has at most $h$ edges, and hence
$A(M)\le h\ell$.

We claim that $\ell\ge(1-o(1))q\log\log m$. The claim is immediate if
$\ell\ge q\log\log m$, so suppose that $\ell<q\log\log m$. Then $X<\log\log m$ and
$A+D\le h\ell+\ell\le2h\ell<2hq\log\log m$.
Hence Lemma~\ref{lem:blowup-matching} gives
$X\ge(1-o(1))\log\log m$.
Hence
$$
\ell\ge qX\ge(1-o(1))q\log\log m,
$$
as required.

The blow-up uses the $q$ base colours, at most $\Delta(J)+1$ original
auxiliary colours, and one additional colour inside the clusters. By
Lemma~\ref{lem:construction-properties}\ref{item:J-degree}, for all
sufficiently large $r$ their total number is at most
$q+\Delta(J)+2\le r$.

Finally, choose $h$ sufficiently large that $N=hn>n_0$. The resulting
$r$-edge-coloured $K_N$ has arbitrarily large order and requires at
least $(1-o(1))r\log\log r$ monochromatic paths, as required.
\end{proof}

\subsection{Monochromatic path coverings with linear deficiency}
\label{sec:robust-covering}

We now prove Theorem~\ref{thm:path-uncovered}. The argument uses the
same blow-up construction as in the proof of
Theorem~\ref{cor:arbitrarily-large}. 
\begin{proof}[Proof of Theorem~\ref{thm:path-uncovered}]
Let $r$ be sufficiently large and set
$q=\lfloor r-r/(\log r)^2\rfloor$. Fix a realization of
Construction~\ref{con:colouring} satisfying
Lemmas~\ref{lem:construction-properties} and~\ref{lem:label-properties}.
For a positive integer $h$, consider Construction~\ref{con:blowup} and
set $N=hn$. As in the proof of Theorem~\ref{cor:arbitrarily-large},
the resulting colouring of $K_N$ uses at most $r$ colours.

Let $\mathcal P$ be a collection of $\ell\le r$ vertex-disjoint
monochromatic paths, and suppose that $\mathcal P$ leaves $u$ vertices
uncovered. Choose a maximum matching in each path and let $M$ be the
union of these matchings. With $A$, $D$, and $X$ as in the preceding
subsection, the same argument as in the proof of
Theorem~\ref{cor:arbitrarily-large} gives $qX\le\ell$ and
$A\le h\ell$. Since a maximum matching in a path leaves at most one
vertex unmatched, we also have $D\le u+\ell$.

We claim that $\mathcal P$ leaves at least $h$ vertices uncovered.
Suppose that $u<h$. Then $X\le r/q=1+o(1)<\log\log m$, while
$
A+D\le h\ell+u+\ell<hr+h+r.
$
Since $q=(1-o(1))r$ and $\log\log m\to\infty$, for all sufficiently
large $r$ we have $hr+h+r\le2hq\log\log m$. Hence
Lemma~\ref{lem:blowup-matching} gives
$X\ge(1-o(1))\log\log m$, contradicting $X\le r/q=1+o(1)$.
Therefore every such collection $\mathcal P$ leaves at least $h$
vertices uncovered.

Since $h$ can be chosen arbitrarily, $N=hn$, and $n$ depends only on
$r$, taking $\eta_r=1/n$ proves the theorem.
\end{proof}

\subsection{Partitions into monochromatic regular graphs}
\label{sec:regular}

We now prove Proposition~\ref{prop:regular}. The argument is a direct
consequence of the matching analysis used for Theorem~\ref{thm:main}.

\begin{proof}[Proof of Proposition~\ref{prop:regular}]
Let $r$ be sufficiently large and set
$q=\lfloor r-r/(\log r)^2\rfloor$. Fix a realization of
Construction~\ref{con:colouring} satisfying
Lemmas~\ref{lem:construction-properties} and~\ref{lem:label-properties}.
Fix $k\ge2$, and suppose that the vertex set is partitioned into
$\ell$ connected monochromatic $k$-regular graphs and single vertices.
Every auxiliary colour class is a matching by
Lemma~\ref{lem:construction-properties}\ref{item:auxiliary-matching},
so every nontrivial $k$-regular part has a base colour. Moreover, by
Lemma~\ref{lem:construction-properties}\ref{item:F-disjoint}, such a
part lies in $F_{i,c}$ for a unique pair $(i,c)$. Since $F_{i,c}$ is
bipartite, every $k$-regular part has a perfect matching.

Choose a perfect matching in each $k$-regular part and let $M$ be their
union. Set $A=|M\cap E(J)|$, $D=n-2e(M)$, and
$X=\frac1q\sum_{i=1}^m|C_i(M)|$. Then $A=0$, while $D$ is the number
of singleton parts. Furthermore, each $k$-regular part contributes to
at most one pair $(i,c)$. Hence $D\le\ell$ and $qX\le\ell$.

We claim that $\ell\ge(1-o(1))q\log\log m$. The claim is immediate if
$\ell\ge q\log\log m$, so suppose that $\ell<q\log\log m$. Then $X<\log\log m$ and
$A+D=D\le\ell<2q\log\log m$, so
Lemma~\ref{lem:matching-support} gives
$X\ge(1-o(1))\log\log m$.
Hence
$$
\ell\ge qX\ge(1-o(1))q\log\log m,
$$
as required.

Finally, since $q=(1-o(1))r$ and
$\log\log m=(1-o(1))\log\log r$, we obtain
$
p_k(r)\ge(1-o(1))r\log\log r.
$
\end{proof}

\section{Balanced complete bipartite graphs}\label{sec:bipartite}
The construction is simpler in this bipartite setting. The two sides of the bipartition play the roles of the
earlier and later endpoints from the complete-graph construction, so
there is no need for an exceptional graph $J$. 

Let $q$ be sufficiently large and set
$m=\lfloor q^{1/5}/\log q\rfloor$ and
$T=\lceil q m^3(\log m)^2\rceil$. Take pairwise disjoint sets
$V_1,\ldots,V_m,W_1,\ldots,W_m$, each of size $T$, and put
$V=V_1\mathbin{\dot\cup}\cdots\mathbin{\dot\cup}V_m$ and
$W=W_1\mathbin{\dot\cup}\cdots\mathbin{\dot\cup}W_m$. Thus
$|V|=|W|=mT$.
Independently for each $y\in W$, choose a uniformly random injection
$\phi_y:[m]\hookrightarrow[q]$. 

We first record the analogue of
Lemma~\ref{lem:label-properties}\ref{item:uniformity2}. Set $\eps=1/(m\log m)$. For $a\in[m]$,
$I\subseteq[m]$, and a family $\mathcal D=(D_i)_{i\in I}$ of subsets
of $[q]$, let $Z(a,I,\mathcal D)$ be the number of vertices
$y\in W_a$ such that $\phi_y(i)\notin D_i$ for every $i\in I$.
 As before, let $\Phi:[m]\hookrightarrow[q]$ denote a uniformly random
injection. We will also use
Lemma~\ref{lem:label-properties}\ref{item:avoidance}, which depends only
on the distribution of $\Phi$.
\begin{lemma}\label{lem:bip-uniformity}
With high probability, simultaneously for all choices of
$a$, $I$, and $\mathcal D$,
$$
\left|
\frac{Z(a,I,\mathcal D)}{T}
-\Pr\bigl(\Phi(i)\notin D_i\text{ for every }i\in I\bigr)
\right|
\le\eps.
$$
\end{lemma}

\begin{proof}
This is the same argument as in the proof of
Lemma~\ref{lem:label-properties}\ref{item:uniformity2}.
The variable $Z(a,I,\mathcal D)$ is the sum of $T$ independent
Bernoulli random variables with common mean
$\Pr(\Phi(i)\notin D_i\text{ for every }i\in I)$.
Hence Hoeffding's inequality and a union bound prove the lemma.
\end{proof}

Fix a realization of the labels for which
Lemma~\ref{lem:bip-uniformity} holds. We now define the colouring.

\begin{construction}\label{con:bip-colouring}
For every $i\in[m]$, every $x\in V_i$, and every $y\in W$, colour the
edge $xy$ with $\phi_y(i)$.
For $i\in[m]$ and $c\in[q]$, let
$
E(F_{i,c})
=
\{xy:x\in V_i,\ y\in W,\ \phi_y(i)=c\},
$
and let $V(F_{i,c})$ consist of the endpoints of these edges.
\end{construction}
By definition, the graphs $F_{i,c}$ partition the edges of $K_{V,W}$.
Moreover, for each fixed $c\in[q]$, the colour-$c$ subgraph is the
union $F_{1,c}\cup\cdots\cup F_{m,c}$.
\begin{lemma}\label{lem:bip-construction}
Construction~\ref{con:bip-colouring} has the following properties. \begin{enumerate}[label=\textup{(\roman*)}]

\item\label{item:bip-F-complete}
For every $i\in[m]$ and $c\in[q]$, the edge set of $F_{i,c}$ is that
of the complete bipartite graph between $V_i$ and
$
\{y\in W:\phi_y(i)=c\}.
$
\item\label{item:bip-F-disjoint}
For each fixed $c\in[q]$, the graphs
$F_{1,c},\ldots,F_{m,c}$ are pairwise vertex-disjoint. 
\end{enumerate}
\end{lemma}
\begin{proof}
Property~\ref{item:bip-F-complete} follows directly from the definition.
For \ref{item:bip-F-disjoint}, fix $c\in[q]$. The sets
$V_1,\ldots,V_m$ are pairwise disjoint. If some $y\in W$ belongs to
both $F_{i,c}$ and $F_{j,c}$ for distinct $i,j$, then
$\phi_y(i)=\phi_y(j)=c$, contradicting the injectivity of $\phi_y$.
\end{proof}

For a matching $M$ in $K_{V,W}$ and $i\in[m]$, let
$
C_i(M)=\{c\in[q]:M\cap E(F_{i,c})\ne\varnothing\}$ and $
X(M)=\frac1q\sum_{i=1}^m|C_i(M)|.
$

\begin{lemma}\label{lem:bip-path-matching}
Suppose that Construction~\ref{con:bip-colouring} is
partitioned into $\ell$ monochromatic paths. Then there is a perfect
matching $M$ such that
$
\sum_{i=1}^m|C_i(M)|\le {3\ell}/{2}.
$
\end{lemma}

\begin{proof}
Choose a maximum matching in each path, and let $M_0$ be their union.
By Lemma~\ref{lem:bip-construction}\ref{item:bip-F-disjoint}, every
monochromatic path containing an edge lies in $F_{i,c}$ for a unique
pair $(i,c)$. Hence
$\sum_{i=1}^m|C_i(M_0)|\le\ell$.

Since $|V|=|W|$ and every edge of $M_0$ meets each side once, the
numbers of vertices left unmatched by $M_0$ in $V$ and $W$ are equal;
denote this common number by $s$. A maximum matching in a path
leaves at most one vertex unmatched, so $2s\le\ell$.
Pair the unmatched vertices of $V$ and $W$ arbitrarily and add the
corresponding $s$ edges to $M_0$. The resulting matching $M$ is
perfect. Each added edge belongs to one graph $F_{i,c}$ and therefore
increases $\sum_{i=1}^m|C_i(M)|$ by at most one. Consequently,
$
\sum_{i=1}^m|C_i(M)|
\le\ell+s\le {3\ell}/{2}.
$
\end{proof}

\begin{proposition}\label{prop:bip-support}
Every perfect matching $M$ in Construction~\ref{con:bip-colouring}
satisfies
$
X(M)\ge {(\log m)}/{16}
$.
\end{proposition}
\begin{proof}
Suppose for a contradiction that $X(M)<(\log m)/16$.
Let $L=\{i\in[m]:|C_i(M)|>q/4\}$ and set $I=[m]\setminus L$. Then
$|L|<4X(M)\le\log m/4$.
Consider the
probability
$
\rho=\Pr\bigl(\Phi(i)\notin C_i(M)\text{ for every }i\in I\bigr).
$
For every $i\in I$, we have $|C_i(M)|\le q/4$. Since $m=o(q)$, for all
sufficiently large $q$,
$
|C_i(M)|/(q-m)<1/3
$
for every $i\in I$. Thus, by
Lemma~\ref{lem:label-properties}\ref{item:avoidance} and the inequality
$\log(1-x)\ge-3x/2$ for $0\le x\le1/3$, we obtain
$$
\begin{aligned}
\rho
&\ge
\prod_{i\in I}\left(1-\frac{|C_i(M)|}{q-m}\right) \\
&\ge
\exp\left(-\frac{3}{2(q-m)}
\sum_{i\in I}|C_i(M)|\right) \\
&\ge e^{-2X(M)}
\ge m^{-1/8}.
\end{aligned}
$$

Let $U$ be the set of vertices $y\in W$ such that
$\phi_y(i)\notin C_i(M)$ for every $i\in I$. For every $a\in[m]$, applying Lemma~\ref{lem:bip-uniformity} with
$D_i=C_i(M)$ for $i\in I$ gives
$
|U\cap W_a|\ge T(\rho-\eps).
$ Summing over $a\in[m]$ gives
$
|U|\ge mT(\rho-\eps).
$

On the other hand, let $y\in U$, and suppose that $y$ is matched by
$M$ to some $x\in V_i$. The edge $xy$ belongs to
$F_{i,\phi_y(i)}$, and hence $\phi_y(i)\in C_i(M)$. By the definition of
$U$, this is impossible when $i\in I$. Therefore every vertex of $U$ must be matched into
$\bigcup_{i\in L}V_i$, and hence
$|U|\le\sum_{i\in L}|V_i|=|L|T$. Combining this with the lower bound
for $|U|$ gives
$$
\rho\le\frac{|L|}{m}+\eps
\le\frac{\log m}{4m}+\frac{1}{m\log m}.
$$
This contradicts $\rho\ge m^{-1/8}$ for sufficiently large $m$.
\end{proof}

\begin{proof}[Proof of Theorem~\ref{thm:bip}]
Let $\ell$ be the number of parts in a monochromatic path partition of
Construction~\ref{con:bip-colouring}. By
Lemma~\ref{lem:bip-path-matching}, there is a perfect matching $M$ such
that
$
qX(M)\le {3\ell}/{2}.
$
Proposition~\ref{prop:bip-support} therefore gives
$
{3\ell}/{2}\ge qX(M)\ge {q\log m}/{16},
$
and hence $\ell\ge q\log m/24$.

Since $m=\lfloor q^{1/5}/\log q\rfloor$, we have
$\log m=(1/5+o(1))\log q$. Thus $\bpath(q)=\Omega(q\log q)$ for every sufficiently large $q$,
which proves Theorem~\ref{thm:bip}.
\end{proof}

\section{Concluding remarks and open problems}
We conclude with several open problems. The most important one is to
determine the correct orders of magnitude of $\ppath(r)$ and
$\pcyc(r)$. We conjecture that both are of order $r\log r$.

\begin{conjecture}\label{conj:order}
We have
$
\ppath(r)=\Theta(r\log r)
\quad\text{and}\quad
\pcyc(r)=\Theta(r\log r).
$
\end{conjecture}
It is also interesting to understand how far apart these two parameters
can be.

\begin{conjecture}\label{conj:path-cycle-ratio}
There exists an absolute constant $C>0$ such that
$\pcyc(r)\le C\ppath(r)$ for every $r\ge2$.
\end{conjecture}

We expect that there exists an absolute constant $C$ such that, for
every fixed $r$, if every $r$-edge-coloured complete graph can be
partitioned into at most $k$ monochromatic paths, then every
sufficiently large $r$-edge-coloured complete graph can be partitioned
into at most $Ck$ monochromatic cycles. A possible approach is to first
extract the matching structure induced by a path partition and transfer
it from complete graphs to almost complete graphs, following the
approach of Letzter~\cite{Letzter}. After applying the Regularity Lemma, this should
yield an almost-spanning matching in the reduced graph supported on at
most $k$ monochromatic components. These components should then give at
most $k$ monochromatic connected matchings, which can be lifted to
cycles by the connected-matching method. Finally, the absorption method
of Gy\'arf\'as, Ruszink\'o, S\'ark\"ozy, and Szemer\'edi~\cite{GRSS}
should cover the remaining vertices and complete the cycle partition.

\smallskip
Another natural question is to determine the smallest number of colours
for which Gy\'arf\'as's conjecture fails. Pokrovskiy~\cite{Pok} showed
that the corresponding cycle-partition conjecture of Erd\H{o}s,
Gy\'arf\'as, and Pyber already fails for $r=3$, whereas he proved that
$\ppath(3)=3$. Our result shows that Gy\'arf\'as's conjecture fails for
all sufficiently large $r$, but does not determine the first such
value.

\begin{problem}\label{prob:smallest-r}
Determine the smallest $r$ for which $\ppath(r)>r$. In particular, is
$
\ppath(4)>4?
$
\end{problem}

\section*{Acknowledgements}
H.L. was supported by the National Natural Science Foundation of China
(12501487), the China Scholarship Council, and the Institute for Basic
Science (IBS-R029-C4). L.W. was supported by the
National Key R\&D Program of China under grant number 2024YFA1013900,
the National Natural Science Foundation of China under grant number
12471327, the China Scholarship Council, and the Institute for Basic
Science (IBS-R029-C4). The authors used AI to assist in working out their ideas and improving the grammar and presentation of the manuscript. All mathematical arguments and proofs in the final manuscript were written by the authors.

\bibliographystyle{amsplain}
\bibliography{reference}

\end{document}